\documentclass[12pt]{amsart}
\usepackage{amssymb,amscd,verbatim, bbold}
\usepackage[all]{xy}
\usepackage[colorlinks,linkcolor=blue,citecolor=blue,urlcolor=red]{hyperref}
\usepackage[symbol]{footmisc}

\title{Mixed motives}
\author{Luca Barbieri-Viale}
\address{Dipartimento di Matematica ``F. Enriques", Universit{\`a} degli Studi di Milano\\  Via C. Saldini, 50\\ I-20133 Milano\\ Italy }
\email{luca.barbieri-viale@unimi.it}


\swapnumbers
\newtheorem{thm}{Theorem}[subsection]
\newtheorem{propose}[thm]{Proposition}
\newtheorem{lemma}[thm]{Lemma}

\theoremstyle{definition}
\newtheorem{defn}[thm]{Definition}

\newtheorem{hyp}[thm]{Hypothesis}
\newtheorem{conjecture}[thm]{Conjecture}
\newtheorem{remark}[thm]{Remark}
\newtheorem{remarks}[thm]{Remarks}
\newtheorem{example}[thm]{Example}

\newcommand{\N}{\mathbb{N}}
\newcommand{\Z}{\mathbb{Z}}     
\newcommand{\Q}{\mathbb{Q}}     
 
\newcommand{\C}{\mathbb{C}}     
     
\newcommand{\Aff}{\mathbb{A}}   
\newcommand{\un}{\mathbb{1}}

\newcommand{\Spec}{\operatorname{Spec}} 
 
\newcommand{\Ex}{\operatorname{\sf Ex}}      
\newcommand{\Cat}{\operatorname{\sf Cat}}

\newcommand{\DM}{\operatorname{DM}}          
    
\newcommand{\MW}{\operatorname{MW}}

\newcommand{\cNM}{\mathcal{NM}}  
\newcommand{\cMW}{\mathcal{MW}}  
\newcommand{\Alb}{\operatorname{Alb}}     

\newcommand{\Sm}{\operatorname{Sm}}
\newcommand{\Sch}{\operatorname{Sch}}
\newcommand{\SmProj}{\operatorname{Sm}^{\operatorname{proj}}}
\newcommand{\Corr}{\operatorname{Corr}}

\newcommand{\End}{\operatorname{End}}      
\newcommand{\car}{\operatorname{char}}
\newcommand{\Tr}{\operatorname{Tr}}

\renewcommand{\L}{\mathbb{L}}
\newcommand{\G}{\mathbb{G}}
\renewcommand{\P}{\mathbb{P}}     

\newcommand{\im}{\operatorname{Im}}        
\renewcommand{\ker}{\operatorname{Ker}}  
\newcommand{\gr}{\operatorname{gr}}        

\newcommand{\longby}[1]{\stackrel{#1}{\longrightarrow}}

\newcommand{\iso}{\longby{\sim}}

\renewcommand{\tilde}{\widetilde}

\newcommand{\ie}{{\it i.e. }}
\newcommand{\cf}{{\it cf. }}
\newcommand{\eg}{{\it e.g. }}

\newcommand{\loccit}{{\it loc. cit. }}
\newcommand{\et} {{\operatorname{\acute{e}t}}}

\newcommand{\add}{{\operatorname{add}}}
\newcommand{\eff}{{\operatorname{eff}}}
\newcommand{\rig}{{\operatorname{rig}}}
\newcommand{\gm}{{\operatorname{gm}}}

\newcommand{\rat}{{\operatorname{rat}}}
\newcommand{\op}{{\operatorname{op}}}

\newcommand{\ab}{{\operatorname{ab}}}

\newcommand{\into}{\hookrightarrow}

\renewcommand{\lim}{\varprojlim}

\newcommand{\boxtensor}{\def\boxtimesten{\Box\kern-7.59pt\raise1.2pt
\hbox{$\times$} }}                                  

\newcounter{elno}                   

\newcommand{\cA}{\mathcal{A}}
\newcommand{\cB}{\mathcal{B}}

\newcommand{\cI}{\mathcal{I}}

\newcommand{\cM}{\mathcal{M}}

\newcommand{\cS}{\mathcal{S}}
\newcommand{\cT}{\mathcal{T}}
\newcommand{\cU}{\mathcal{U}}

\newcommand{\cW}{\mathcal{W}}

\renewcommand{\phi}{\varphi}
\renewcommand{\epsilon}{\varepsilon}

\begin{document}
\begin{abstract}
A mixed Weil cohomology with values in an abelian rigid tensor category is a cohomological functor on Voevodsky's category of motives which is satisfying K\"unneth formula and such that its restriction to Chow motives is a Weil cohomology. We show that the universal mixed Weil cohomology exists. Nori motives can be recovered as a universal enrichment of Betti cohomology via a localisation. This new picture is drawing some consequences with respect to the theory of mixed motives in arbitrary characteristic.
\end{abstract}
\maketitle

\hfill {\small \it Dedicated to Jaap Murre}\\

\section{Introduction}
The notion of generalised Weil cohomology is conceived in \cite{BVKW} showing that the universal theory exists for smooth projective varieties over any field or suitable base scheme. It is then natural to seek for a corresponding notion of mixed Weil cohomology for algebraic or arithmetic schemes and show that also the universal one exists. 

A construction of the universal mixed Weil cohomology is the main task of this paper along with the expected picture arising from a theory of mixed motives: in the latter purpose this paper is also the natural continuation of \cite{BV}. 

Note that mixed Weil cohomologies have been considered by Cisinski-D\'eglise \cite{CD} and Ayoub \cite{AW} and \cite{AWH} but we here follow what just hinted by Andr\'e \cite[\S 14.2.4]{AM} and develop it in our general framework of generalised cohomologies, especially, that of \cite[Def.\,4.2.1]{BVKW}. Since Weil cohomologies should be recovered from the mixed as pure, the key requirement for a mixed Weil cohomology is that its restriction to smooth projective varieties is a Weil cohomology. 

\subsection{Mixed versus pure} To set a general target for our cohomology theories we just work with an abelian rigid tensor $\Q$-linear category $(\cA, \otimes, \un)$ together with a Lefschetz object $L\in\cA$ for which we write  $A(q):= A\otimes L^{\otimes -q}$ for any $A\in \cA$ and any $q\in \Z$. Recall that a Weil cohomology with values in $(\cA, L)$ is nothing else than a tensor $\Q$-linear graded functor
$$H^*: \cM_\rat\to \cA^{(\Z)}$$
from the rigid category $\cM_\rat$ of Chow motives to that of finitely supported graded objects $\cA^{(\Z)}$ satisfying effectivity and trace conditions, \eg it is required the existence of an isomorphism $\Tr: H^2(\L)\iso L$ for $\L\in\cM_\rat$ the Lefschetz Chow motive (see \cite[Prop.\,4.4.1]{BVKW} for details). 

The key idea is that Voevodsky motives must play the same rôle for mixed Weil cohomologies as Chow motives do for Weil cohomologies. In fact, the opposite category of Chow motives $\cM_\rat^\op$ can be regarded as a full tensor subcategory of Voevodsky's category $\DM_\gm$, the triangulated rigid tensor $\Q$-linear category of geometric motives (see \cite{V} and \cite[Thm.\,18.3.1.1]{AM}). 
Therefore, it is natural to consider tensor $\Q$-linear graded functors
$$H^*: \DM_\gm^\op\to \cA^{(\Z)}$$
which are arising from an effectively bounded cohomological functor $H$ (Definition \ref{bound}) and such that their restriction to $\cM_\rat$ are Weil cohomologies: this is our actual definition of mixed Weil cohomology (Definition \ref{mixweil}). Remark that as $\L\in \cM_\rat$ is sent to $\Z (1)[2]\in\DM_\gm$ we have that $H^*(\Z(q)[2q])$ is isomorphic to $\un (-q)=L^{\otimes q}$ concentrated in degree $2q$ for any $q\in \Z$. 

By the way, our setting suffices to include Cisinski-D\'eglise mixed Weil cohomologies via a realisation functor (as in \cite[Thm.\,3]{CD}, see Example \ref{ex}) and, in particular, the Grothendieck class given by $\ell$-adic relative cohomologies of algebraic schemes defined over a field $k$, in any characteristic with $\car k\neq \ell$, Betti and de Rham cohomologies in zero characteristic, and rigid cohomology in positive characteristics, nowadays called classical mixed Weil cohomologies. A paradigmatic non-classical example of mixed Weil cohomology is that of the so called de Rham-Betti cohomology over the field $\bar \Q$ of algebraic numbers (see \cite[\S 7.1.6]{AM}, \cite[\S 2]{HU}, \cite[Chap.\,5]{HMN} and \cite{ABVB}) or the absolutely flat version of Ayoub's new Weil cohomology (see \cite[Ex.\,3.2]{BV}). 

\subsection{Universality} The key properties of a mixed Weil cohomology are direct consequences of that of Voevodsky motives: we get a relative cohomology (Lemma \ref{rel}) satisfying a K\"unneth formula for pairs and a relative duality (Lemma \ref{dual}). 

Notably, these cohomology theories can be pushed forward along exact strong tensor functors and we can show that the induced 2-functor of mixed Weil cohomologies is representable (Theorem \ref{mixed}). This result is providing a cohomological functor $MW$ whose target category is denoted $\cMW$, which is endowed with a tautological Lefschetz object, and a mixed Weil cohomology
$$MW^*: \DM_\gm^\op\to \cMW^{(\Z)}$$ 
such that for any mixed Weil cohomology $H$ with values in $(\cA, L)$ there exists a unique exact tensor functor $$\Psi_H: \cMW\to \cA$$ such that $\Psi_HMW^p\cong H^p$ for all $p\in\Z$ compatibly with the Lefschetz objects. 
We also obtain the tight variant $MW^+$ of $MW$ (Theorem \ref{tight}) with target category denoted $\cMW^+$, which is a quotient of $\cMW$ by a Serre tensor ideal, in order to impose weak and hard Lefschetz, Albanese invariance and normalisation (Definition \ref{mixtightweil}). The functor $\Psi_H$ factors through $\cMW^+$ providing the classifying functor $$\Psi_H^+:\cMW^+\to \cA$$ if and only if $H$ is tight. 

Moreover, for any mixed Weil cohomology $H$ with values in $(\cA, L)$ we obtain its universal enrichment $\cMW_H$ (Theorem \ref{initial}), again obtained as a quotient of $\cMW$, by the kernel of $\Psi_H$ this time, such that $H$ is now a realisation (Definition \ref{equi}) along the induced faithful tensor functor
$$\bar\Psi_H: \cMW_H\into \cA$$
and universal with respect to this property: the pushforward of the universal theory along the quotient is the universal enrichment of $H$ and it is tight if $H$ is. 

Remark that these quotient abelian categories of a tensor exact abelian category inherit a unique tensor structure such that the projection is a tensor functor and this tensor structure is exact. If we started with a rigid category its quotient remains rigid (see \cite[Prop.\,4.5]{BVK}).

As an application, for Betti cohomology $H=H_{\rm Betti}$, we can see that the abelian rigid tensor category $\cMW_H$ is tensor equivalent to Nori motives (Theorem \ref{Nori}) rebuilding it by a straightforward construction without recourse to quivers, good pairs nor basic lemma for the tensor structure. 

In general, we further show a simple way of constructing regulator mappings (Lemma \ref{reg}) which yields a universal regulator map 
$$r^{p,q}: H^{p, q}_{\rm M}(X)\to \cMW (\un, MW^{p}(X)(q))$$
where $H^{p, q}_{\rm M}(X)$ is Voevodsky's motivic cohomology of any $X$ algebraic scheme. This yields a universal cycle class map
$$c\ell^q: CH^{q}(X)_\Q\to\cMW (\un,MW^{2q}(X)(q))$$
for $X$ smooth and projective.

To compare with the pure case, recall the existence of the universal tight Weil cohomology $W_\ab^+$ with values in the abelian rigid category $\cW_\ab^+$ (previously constructed in \cite[Thm.\,8.4.1]{BVKW}); by construction, since the restriction of a mixed (tight) Weil cohomology to Chow motives is a (tight) Weil cohomology (Lemmas \ref{weil}, \ref{pure} and Theorem \ref{tight}), we obtain a comparison exact tensor functor
$$\Phi^+: \cW_\ab^+\to \cMW^+$$
whose essential image is contained in the stricly full subcategory $\cMW^+_{\rm pure}$ of $\cMW^+$ given by the minimal abelian tensor subcategory containing the universal cohomology of smooth projective varieties (Lemma \ref{purification}).

\subsection{Motivic pictures} Following Andr\'e \cite[\S 0.3]{AB} a category of mixed motives over a subfield of the complex numbers, is a $\Q$-linear Tannakian category \cite{DT} for which Betti cohomology is a fibre functor, which is the target of a universal cohomology theory and whose morphisms are of geometric origin, that is, they should be given by correspondences. Voevodsky's $\Q$-linear theory of mixed motives over a field of arbitrary characteristic, as stated in \cite[\S 4]{VN}, is a stronger formulation of these principles and it provides $\DM_\gm$ as the bounded derived category of this Tannakian category (up to equivalences) if such a category exists. 

The existence of the universal mixed (tight) Weil cohomology suggests a new unconditional context from which we can look at a theory of mixed motives. However, the construction of this universal theory doesn't grant that the target abelian rigid category $\cMW^+$ is Tannakian: a priori, it could well be given by a product of Tannakian categories! This depends on the simplicity of the unit object, only! Moreover, this property is equivalent to the fact that all mixed (tight) Weil cohomologies are equivalent in the sense of Definition \ref{equi} and Theorem \ref{mixmot}. 

For pure motives, the standard hypothesis that $\cW_\ab^+$ is Tannakian is a consequence of Grothendieck standard conjectures, see \cite[Cor.\,3.9 \& Hyp.\,3.10]{BV}. Therefore, following Andr\'e and Voevodsky, the hypothesis that $\cMW^+$ is Tannakian (Hypothesis \ref{standard} and Remarks \ref{rmkstd}) is the corresponding mixed analogue. If the named comparison functor $\Phi^+$ is fully faithful these two hypotheses are equivalent and, over a subfield of the complex numbers, these hypotheses imply that the essential image of $\Phi^+$ is $\cMW^+_{\rm pure}$ indentified with Andr\'e motives inside $\cMW^+$ in turn identified with Nori motives (Proposition \ref{prop}). In particular, $\cMW^+_{\rm pure}$ is semisimple and $MW^+$ is cellular in this case (by \cite[Thm.\,0.4]{A} and \cite[\S 9.2]{HMN}, respectively). 

Unconditionally, it seems plausible that $\Phi^+$ is fully faithful with essential image $\cMW^+_{\rm pure}$ split (Conjecture \ref{conj}) even if the source and target categories of $\Phi^+$ were not Tannakian and, additionally, we may show that the universal relative cohomology $MW^+$ is always cellular, even in positive characteristics, and represented (as in Example \ref{ex}) by a
triangulated tensor functor
$$R^+:\DM_\gm^\op\to D^b(\cMW^+)$$
as hinted in Remark \ref{basic} akin to the situation with Nori cohomological motives (modulo duality, compare with \cite[Prop.\,7.12]{CG} and \cite{HMN}). Cellularity would yields representability for all mixed tight Weil cohomologies via the composition $R^+_H:=R\Psi_H^+ R^+$ where 
$$R\Psi_H^+:  D^b(\cMW^+)\to  D^b(\cA)$$
is the canonical derived functor of the named $\Psi_H^+$ which is exact. Whether the functor $R^+$ shall be an equivalence or not is just related to the geometric origin of the universal theory.

Finally, $MW^+$ cellular with $\cMW^+$ Tannakian would also be the universal theory of motivic type in the sense of Voevodsky, see \cite[\S 4]{VN} for details: just recall that such theories are given by a Tannakian category $\cA$ and a functor $M:\Sch_k^\op\to D^b(\cA)$, along with a comparison with $\ell$-adic cohomology. Voevodsky claims that any theory of motivic type can be extended to a triangulated tensor functor $R_M:\DM_\gm\to D^b(\cA)$ and, actually, as in Example \ref{ex}, this yields a mixed Weil cohomology, say $H$, that would be an enrichment of $\ell$-adic cohomology and $\cMW^+\cong \cMW_H$ for any such $H$ which is also tight. Thus it yields $\Psi_{H}^+$ whence a factorisation of $R_M$ through $D^b(\cMW^+)$ via $R^+$ as above.

\subsubsection*{Acknowledgements} 
I would like to thank IHES for providing support, hospitality and excellent working conditions. I'm glad to thank Bruno Kahn for several suggestive discussions on almost all matters contained in this paper.

\subsubsection*{Notations and assumptions} Adopt current conventions on small and large categories, considering a fixed universe when small or locally small is specified. For abelian categories $\cA$, we say that a full abelian subcategory $\cB\subseteq \cA$ is \emph{generated} by a set of objects $\{A_i\mid A_i\in \cA\}_{i\in I}$ if it is the minimal abelian full subcategory $\cB$ such that $A_i\in \cB$, just given by the subquotients of $A_i$ for all $i\in I$. Say that $\cB$ is a \emph{Serre or thick} subcategory of an abelian category $\cA$ if it is closed by subobjects, quotients and extensions. Refer to the abelian category  $\cA/\cB$ as the Serre \emph{quotient} category. Similarly, a \emph{thick} subcategory of a triangualted category is a  full triangulated subcategory closed under direct summands. 
For $(\cA, \otimes, \un)$  \emph{$\otimes$-abelian}  we mean an abelian $\Q$-linear category endowed with an \underline{exact} symmetric unital monoidal structure. In this case we refer to a Serre \emph{ideal} for a Serre subcategory which is also a tensor ideal providing the quotient $\cA/\cS$ with the induced tensor structure (see \cite[Prop.\,4.5]{BVK}).  Following \cite[Appendix 8A]{VL} \emph{$\otimes$-triangulated} category $(\cT, [\ \ ],  \otimes, \un)$ is an additive category which is a triangulated category and a tensor category together with natural isomorphisms $r: -\otimes +[1]\iso (-\otimes +)[1]$ and $l : -[1]\otimes +\iso (-\otimes +)[1]$  which commute with the associativity, commutativity and unity isomorphisms. Note that under the symmetry iso $\sigma: M\otimes N\iso N\otimes M$ we have that $\sigma l\sigma = r$. A \emph{triangulated tensor} functor is a strong symmetric monoidal functor which is unital, additive, commutes with translation and preserves the distinguished triangles.

\section{Cohomological functors} 
\subsection{Voevodsky motives} \label{1.1}
Let $\Sch_k$ be the category of schemes which are separated and of finite type over a fixed base field $k$. Let $\Sm_k$ be the full subcategory of $\Sch_k$ given by smooth schemes and $\SmProj_k$ that of smooth and projective varieties.  

Let $\Corr_k$ be the category of Voevodsky correspondences: same objects as $\Sm_k$ but morphisms are finite
correspondences and compositions of morphisms are compositions of correspondences. 
We have a covariant functor  $[-] : \Sm_k\to \Corr_k$ sending a morphism $f : X \to Y$ to its graph, as a finite correspondence
from $X$ to $Y$. Recall (see \cite{V} or \cite[\S 16]{AM} for details) Voevodsky's category $$\DM_\gm^\eff:= (K^b(\Corr_k)/\cT)^\natural$$ defined as the pseudo-abelian envelope of the localisation of the homotopy category of bounded complexes $K^b(\Corr_k)$ at the thick subcategory $\cT$ generated by $[X\times \Aff^1_k]\to [X]$ and $[U\cap V]\to [U]\oplus[V]\to [X]$ for any $X\in \Sm_k$ and $X=U\cup V$ open cover. 

We then get a functor $M: \Sm_k\to \DM_\gm^\eff$ following Voevodsky notation in 
\cite[Def. 2.1.1]{V};  considering Chow correspondences we also have the category of Chow effective motives $\cM_\rat^\eff$ along with a contravariant functor $h: \SmProj_k\to \cM_\rat^\eff$ given by the graph of a morphism, see \cite[Def.\,4.1.3]{BVKW} from which we borrow the notation.

There are canonical symmetric monoidal structures on both $\cM_\rat^{\eff}$ and $\DM_\gm^\eff$, $h$ and $M$ are tensor functors 
and we get the following  commutative diagram
$$\xymatrix{\SmProj_k\ar[r]^{} \ar[d]_h & \Sm_k \ar[d]^{M} \\  
\cM_\rat^{\eff,\op} \ar[r]^{\iota^\eff} & \DM_\gm^\eff}
$$
where $\iota^\eff$ is a fully faithful tensor functor (see \cite[Prop.\,2.1.4]{V} and compare with \cite[Thm.\,18.3.1.1]{AM}). We have  $\Z := M(\Spec (k))= \iota^\eff(h(\Spec (k)))$ ( = strict unitality) and any rational point $x\in X(k)$ is providing a splitting $M(X) = \Z \oplus \tilde M(X)$ where $\tilde M(X)$ is the reduced motive (see \cite[\S 2.1]{V} and \cite[\S 16.2.5]{AM}).

Now let $\cM_\rat:= \cM_\rat^\eff[\L^{-1}]$ for  $\L:= h^2(\P^1_k)$ the effective Lefschetz object in $\cM_\rat^\eff$, and $\DM_\gm:= \DM_\gm^\eff[\Z (1)^{-1}]$ where, in Voevodsky's notation, 
$\Z (1):= \tilde M (\P^1_k)[-2]$ is the effective Tate object: here the reduced motive $\tilde M (\P^1_k)$ yields a decomposition $M (\P^1_k) = \Z\oplus \Z (1)[2]$, $\iota^\eff(\L)=\Z (1)[2]$, and $\Z (q):= \Z (1)^{\otimes q}$ in $\DM_\gm$ for any integer $q\in \Z$ (here $\Z (0):= \Z$). We consider $\cM_\rat$ and $\DM_\gm$ with $\Q$-coefficients: these categories are rigid and there is a fully faithful tensor functor 
$$\iota : \cM_\rat^\op \to\DM_\gm$$
extending $\iota^\eff$ above. In the following we shall regard $\cM_\rat$ in $\DM_\gm$ via $ \iota$. 

There is an extension of $M : \Sch_k \to \DM_\gm^\eff$ (with $\Q$-coefficients) and we also have the motive with compact support $M^c(X)$ for any $X\in\Sch_k$ such that $M^c(X)\to M(X)$ is an isomorphism if $X$ is proper and $M^c(X)=M(X)^\star(d)[2d]$ for $X$ smooth equidimensional $d = \dim (X)$, denoting by $(\ \ )^\star$ the dual in $ \DM_\gm$.

In general, set $H^{p, q}_{\rm M}(X) := \DM_\gm (M(X), \Z (q)[p])$, $H_{p}^{\rm SV}(X) := \DM_\gm (\Z [p],M(X))$ and $H_{p, q}^{\rm BM}(X) := \DM_\gm (\Z (q)[p], M^c(X))$
which we here refer to as Voevodsky's motivic cohomology, Suslin-Voevodsky homology and Borel-Moore homology $\Q$-vector spaces 
(see \cite[Def.\,16.20]{VL}, \cite[\S 2.1 \& 4.1]{V} and \cite[\S 16.2.5]{AM}).

\subsection{K\"unneth formula and relative cohomology}  Let $(\cA, \otimes, \un)\in \Ex^\otimes $ be $\otimes$-abelian.
Recall that a $\Q$-linear functor $H: \DM_\gm \to \cA$ is said to be homological if it takes a distinguished triangle $M \to N \to F \longby{+1}$ to an exact sequence $H(M) \to H(N) \to H(F)$ in $\cA$. 
\begin{defn} \label{cohext} Say that a homological functor $$H:\DM_\gm^\op\to \cA$$ is a \emph{cohomological functor}. Same definition applies to $\DM_\gm^\eff$. 
Say that $H$ is endowed with an \emph{external product} if there is a supplementary structure of a lax tensor functor $(H, \kappa, \upsilon)$ where, $\kappa$ is a collection of maps
$$\kappa_{M, N}: H(M) \otimes  H(N) \to H(M\otimes  N)$$
for $M, N\in \DM_\gm$ which are compatible with the associative constraints and the symmetry isomorphism, and  $\upsilon : \un\to H(\Z) $ is a morphism satisfying the unitality condition (compare with \cite[\S 3.2]{BVKW}). 
\end{defn}

For a cohomological functor $H$ endowed with an external product $\kappa$ we then have
$$\kappa_{M, N}^{i,j}: H(M[i])\otimes  H(N[j])\to H(M[i]\otimes N[j])\cong H((M\otimes N)[i+j])$$
given by composition of $\kappa$ with the canonical isomorphisms $M[i]\otimes N[j]\cong (M\otimes N)[i+j]$.
Note that under the symmetry isomorphisms $\tau$ we obtain
$$\xymatrix{H(M[i])\otimes  H(N[j])\ar[d]_-{(-1)^{ij}\tau}\ar[r]^-{\kappa_{M, N}^{i,j}} & H((M\otimes N)[i+j])\ar[d]^-{H(\tau [i+j])}  \\
H(N[j])\otimes  H(M[i]) \ar[r]^-{\kappa_{N, M}^{j,i}}  &H((N\otimes M)[i+j])}$$
coherently with the Koszul constraint (see \cite[Def.\,8A3]{VL} and \cite[Remark 3.2.4]{BVKW}). 
\begin{defn} \label{bound} Let $H$ be a cohomological functor.  Denote $H^p(M):=H(M[p])$, for all $p\in\Z$ and $M\in \DM_\gm$. Set
$$H^{p}(X):= H^{p}(M(X))\hspace{0.2cm}\text{and}\hspace{0.2cm} H_{p}(X):= H^{-p}(M^c(X)^\star)$$ for the cohomology and the (Borel-Moore) homology of $X\in\Sch_k$. 
Say that $H$ is \emph{bounded} if $H^*:= \{H^p\}_{p\in\Z}$ is such that
$$H^*:\DM_\gm^\op \to \cA^{(\Z)}$$
where $\cA^{(\Z)}$ is the tensor category of finitely supported graded objects. Say that it is \emph{effectively bounded} if further $H^p(X) =0$ for $p\notin [0, 2d ]$ for $d=\dim (X)$ and $X\in \Sm_k$ equidimensional.  If $(H, \kappa, \upsilon)$ bounded is endowed with an external product, let $(H^*,\kappa^*, \upsilon^*)$ be the induced graded cohomological functor endowed with the induced graded external product.
\end{defn}
\begin{lemma}\label{kunp}
  Let $(H,\kappa, \upsilon)$ be a bounded cohomological functor endowed with an external product. We then have that $(H^*,\kappa^*, \upsilon^*)$  is strong as a tensor functor if and only if the following are isomorphisms
  $$\kappa^{k}_{M,N}: \sum_{i+j=k} H^i(M)\otimes  H^j(N)\iso H^k(M\otimes N)$$
  for all $k\in\Z$, $M, N\in \DM_\gm$, $\upsilon: \un \iso H(\Z)$ is an isomorphism and $H^i(\Z)=0$ for $i\neq 0$ holds true. Moreover, in this case, $H$ is taking values in the $\otimes$-abelian full subcategory $\cA_\rig\subseteq \cA$ of dualisable objects; in particular, $H^i(M)\in\cA_\rig$ for all $i\in\Z$ and $M\in \DM_\gm$.
\end{lemma}
\begin{proof}
    For strongness see \cite[Remarks 3.2.2 \& 3.2.4]{BVKW} and for rigidity \cite[Lemma 3.3.6 b)]{BVKW} since $\DM_\gm$ is rigid hence $H^*(M)$ is dualisable for all $M\in \DM_\gm$. Since the tensor structure of $\cA$ is exact then its full subcategory $\cA_\rig$ is abelian by \cite[Prop.\,4.1]{BVK}.
\end{proof}
\begin{defn}[\protect{K\"unneth and Point axioms}] \label{kunpdf}
Let $(H,\kappa, \upsilon)$ be a bounded cohomological functor endowed with an external product and let $(H^*,\kappa^*, \upsilon^*)$ be the induced graded one.  Say that $(H,\kappa, \upsilon)$ satisfies the \emph{K\"unneth} axiom if $\kappa^*$ is an isomorphism and the \emph{point} axioms if  $\upsilon^*$ is an isomorphism. 
\end{defn}
Equivalently, from Lemma \ref{kunp}, $H$ as in Definition \ref{kunpdf} is such that $H^*$ in Definition \ref{bound} is a strong tensor functor. Moreover, this yields a K\"unneth formula 
$$\kappa^{k}_{X,Y}: \sum_{i+j=k} H^i(X)\otimes  H^j(Y)\iso H^k(X\times Y)$$
for all $X, Y\in \Sch_k$ since $M(X)\otimes M(Y)\cong M(X\times Y)$ in $\DM_\gm^\eff$ by \cite[Prop.\,2.1.3 \& 4.1.7]{V}. By the way, for any cohomological functor $H$ the cohomology objects are also homotopy invariant $$H^*(X)\iso H^*(\Aff^1_X)$$ by \cite[Cor.\,4.1.8]{V} for any $X\in Sch_k$,  there is a Mayer-Vietoris sequence for open covers, projective bundle decomposition and the following exact sequence for abstract blow ups, \ie given by a proper birational morphism $f:X\to X'$, $Y=  f^{-1}(Y')$ and
$f: X-Y\iso X'-Y'$, 
$$\cdots\to H^{p-1}(Y)\longby{\partial^{p-1}}  H^p(X') \to  H^p(X)\oplus H^p(Y') \to H^p(Y)  \to\cdots $$ 
as a consequence of the properties in \cite[\S 2.2]{V}.

Finally, recall \cite[Def.\,1.1]{BVKD} where $M (X, Y)$ (resp.\, $M^Y (X)$) is the relative motive (resp.\, the motive with support) associated to a pair $(X, Y)$ for $X\in \Sch_k$ and $Y\subseteq X$ closed. Let $\Sch_k^\square$ be the category of such pairs.
\begin{lemma}[\protect{Relative cohomology}] \label{rel}
Let  $H$ be a bounded cohomological functor. Let $H^p(X, Y):= H^p (M(X, Y))$ and  $H^p_Y(X):= H^p (M^Y(X))$ for $p\in \Z$.  Then  $H^*: \Sch_k^\square\to \cA^{(\Z)}$ is a relative cohomology such that for a triple $(X,Y)$, $(X,Z)$ and $(Y,Z)$ we have the long exact sequence 
$$\cdots\to H^{p-1}(Y,Z)\longby{\partial^{p-1}}  H^p(X,Y) \to  H^p(X,Z) \to H^p(Y,Z)  \to\cdots $$ and
this long exact sequence is natural \ie  another triple $(X',Y')$, $(X',Z')$ and $(Y',Z')$ together with morphisms $\gamma: (Y,Z)\to (Y',Z')$, $\varphi: (X,Z)\to (X',Z')$ and $\delta: (X,Y)\to (X',Y')$ is inducing the following commutative diagram in $\cA$ 
$$\xymatrix{\cdots \ar[r]^-{} & H^p(X,Y)\ar[r]^-{}   & H^p(X,Z) \ar[r]^-{} & H^p(Y,Z) \ar[r]^-{\partial^p}  &  H^{p+1}(X,Y)\ar[r]^-{}  & \cdots\\
\cdots \ar[r]^-{} & H^p(X',Y')\ar[r]^-{} \ar[u]^-{\delta^p} & H^p(X',Z') \ar[r]^-{} \ar[u]^-{\varphi^p} & H^p(Y',Z') \ar[r]^-{\partial'^p}  \ar[u]^-{\gamma^p}  &  H^{p+1}(X',Y') \ar[r]^-{}  \ar[u]^-{\delta^{p+1}}  & \cdots}$$
Moreover, the following 
$$\cdots\to H^{p-1}(X-Y)\longby{\partial^{p-1}}  H^p_Y(X) \to  H^p(X) \to H^p(X-Y)  \to\cdots $$ 
is exact in $\cA$ and also natural. If $\delta: (X,Y)\to (X',Y')$ is an abstract blow up we have the excision
$\delta^* :H^*(X',Y')\iso H^*(X,Y)$ isomorphism.  For $(H,\kappa, \upsilon)$ satisfying the K\"unneth axiom and pairs $(X,Y), (X',Y')\in Sch_k^\square$  we have 
$$\kappa^{k}_{(X,Y),(X',Y')}: \sum_{i+j=k} H^i(X,Y)\otimes  H^j(X',Y')\iso H^k(X\times X', Y\times X'\cup X\times Y')$$ 
for all $k\in\Z$.
\end{lemma}
\begin{proof} By the construction in \cite{BVKD} the usual distinguished triangles associated with pairs and triples yield the corresponding long exact sequences in $\cA$. For an abstract blow up we have that $M(X,Y)\cong M(X',Y')$ depends on the open complement only. Since $M(X,Y)\otimes M(X',Y')\cong M(X\times X', Y\times X'\cup X\times Y')$ we obtain the claimed relative K\"unneth formula.
\end{proof}
\begin{remark} \label{weight}
Recall that Bondarko \cite{BO} provided $\DM_\gm$  of a weight structure such that $$w_{\leq i}M\to M \to w_{\geq i+1}M\longby{+1}$$ is a distinguished triangle for all $M\in\DM_\gm$. The interested reader can also be inspecting the previous facts with respect to Bondarko's weights $$W_iH^p(X,Y)\subseteq H^p(X,Y)$$ recalling that for $M \in\DM_\gm$ we have Bondarko's Chow weight filtration
$$W_iH^p(M) := \im(H(w_{\geq i}M[p]) \to H(M[p]))$$
and $W_{i-1}H^p\into W_iH^p$ are well-defined subfunctors of $H^p$ such that
$\gr^W_iH^p(M)=\ker H^{-i}(N)\to H^{-i}(N')$ for $N\to N'\in\cM_\rat$ (see \cite[Prop.\,6.1.2]{HMN} and \cite[\S 2]{BO}).
\end{remark}

\section{Mixed Weil cohomologies}

\subsection{Dualities and regulators} Let  $(\cA, L)$ be $\otimes$-abelian $(\cA, \otimes, \un)$ together with a Lefschetz object $L$, \ie an invertible object of $\cA$. For $A\in\cA$ and $q\in\Z$ we shall denote $A(q):= A \otimes T^{\otimes q}$ where $T:= L^{-1}$ is the corresponding Tate object  (compare with \cite[\S 4.2]{BVKW}). For a cohomological functor $H$ with values in $\cA$ and $(X,Y)\in\Sch_k^\square$ set
$$H^{p, q}(X, Y):= H^p(X,Y)(q)\hspace{0.2cm}\text{and}\hspace{0.2cm} H_{p,q}(X):= H_p(X)(-q)$$ for all $p, q\in \Z$. Remark that we can also define cohomology with compact support $H^{p, q}_c(X):= H^{p}(M^c(X))(q)$ and the usual homology $H_{p, q}^{\rm S}(X):= H^{-p}(M(X)^\star)(-q)$ coinciding with $H^{p, q}(X)$ and $H_{p, q}(X)$, respectively, if $X$ is proper.
\begin{defn}[\protect{Trace axiom}] \label{trace}
A cohomological functor $H$ with values in $\cA$ together with a Lefschetz object $L$ and an additional morphism $\Tr: H^2(\Z(1))\to L$ in $\cA$ is denoted $(H,\Tr)$, is said to have a trace and take values in $(\cA, L)$.

Say that  $(H,\Tr)$ with values in $(\cA, L)$ satisfies the \emph{trace} axiom if $\Tr : H^2(\Z(1))\iso L$ is an isomorphism, $H^i(\Z(1))=0$ for $i\neq 2$ and $\pi^0 : H^{0}(\Z)\iso H^{0}(X)$ 
is an isomorphism if $X$ is geometrically connected, where the canonical morphism $\pi^0$ is induced by the structural morphism $\pi: X\to \Spec (k)$.
\end{defn}
The following notion of mixed Weil cohomology is modelled on the one hinted in Andr\'e's book \cite[\S 14.2.4.]{AM}.
\begin{defn}[Mixed Weil cohomology]  \label{mixweil}
Let  $(\cA, L)$ be a $\otimes$-abelian category together with a Lefschetz object.
A cohomological functor 
$$H:\DM_\gm^\op\to \cA$$ is a \emph{mixed Weil cohomology} with values in $(\cA, L)$ if it is endowed with an external product and a trace $(H, \kappa, \upsilon, \Tr)$ such that $H^*$ is effectively bounded and satisfies K\"unneth, point and trace axioms (see Definitions \ref{cohext}, \ref{bound}, \ref{kunpdf} \& \ref{trace}). We shall write $(\cA, H)$ for a mixed Weil cohomology taking values in $(\cA, L)$ without making explicit all the data if not needed. 
 \end{defn}
\begin{lemma}[\protect{Duality}] \label{dual}
Let $(\cA, H)$ be a mixed Weil cohomology. For $X$ smooth and $p, q\in\Z$ we have the duality isomorphism
$$H^{p, q}(X)\cong H_{2d-p,d-q}(X)$$
where $d=\dim (X)$ for $X$ equidimensional; if $(X,Y)\in\Sch_k^\square$ is such that $X$ is proper and $X-Y$ is smooth we then have the isomorphism
$$H^{p,q}(X,Y)^\vee\cong H^{2d-p, d-q}(X-Y)$$
and for $X$ smooth and proper, $(X, Z)\in\Sch_k^\square$ such that $Y\cap Z =\emptyset$, we have the relative duality isomorphism
$$H^{p, q}(X-Z,Y)^\vee\cong H^{2d-p, d-q}(X-Y, Z)$$
for any $p, q\in\Z$.
\end{lemma}
\begin{proof} For $X$ smooth we have $M^c(X)^\star\cong M(X)\otimes \Z(-d)[-2d]$ and applying $H^*$ we obtain the graded isomorphism 
$H^*(M^c(X)^\star)\cong H^*(M(X)) \otimes H^*(\Z(-d)[-2d])$ where $H^*(\Z(-d)[-2d])$ is $\un (d)$ concentrated in degree $-2d$; this is  providing the claimed duality isomorphism after twisting by $(q-d)$, since $*-2d =-p$ for $*=2d-p$. For $X$ proper and $X-Y$ smooth we have $M(X, Y)\cong M^c(X-Y)\cong M(X- Y)^\star(d)[2d]$ so that $H^*(X,Y)\cong H^*(X- Y)^\vee\otimes H^*(\Z(d)[2d])$ where now $H^*(\Z(d)[2d])$ is $\un (-d)$ concentrated in degree $2d$ thus $(H^{p}(X,Y)(q))^\vee\cong H^{2d-p}(X- Y)(d-q)$ in this case. Finally, for $X$ smooth and proper, $Y,Z\subseteq X$ two disjoint closed subsets we have $M(X-Z,Y)\cong M(X-Y,Z)^\star (d)[2d]$ whence $H^*(X-Z,Y)\cong H^*(X-Y, Z)^\vee\otimes H^*(\Z(d)[2d])$ providing the claimed relative duality isomorphism (see \cite{BVKD}).
\end{proof}
\begin{remark} \label{reldual}
If $X$ is smooth and proper, $d=\dim (X)$, $Y$ and $Z$ are two normal crossing divisors such that the union is also a normal crossing divisor then a more general relative duality holds true
$M(X-Z,Y-Y\cap Z) = M(X-Y,Z-Y\cap Z)^\star (d)[2d]$ as a consequence of the motivic formalism of six functors (as suggested by J. Ayoub). Therefore also $$H^{p, q}(X-Z,Y-Y\cap Z)^\vee\cong H^{2d-p, d-q}(X-Y, Z-Y\cap Z)$$
is available for any $H$ mixed Weil cohomology.
\end{remark}
\begin{lemma}[\protect{Regulators}]\label{reg}
For a mixed Weil cohomology $(\cA, H)$ and any $X\in\Sch_k$ we get regulator maps
$$r^{p, q}_H: H^{p,q}_{\rm M}(X)\to \cA (\un,H^{p, q}(X))\hspace{0.4cm}\text{and}\hspace{0.4cm} r_{p, q}^H: H_{p,q}^{\rm BM}(X)\to \cA (\un,H_{p, q}(X)).$$  
For $X\in\Sch_k$ equidimensional of dimension $d$ and $0\leq q\leq d$, there is a higher cycle class map from Bloch's higher Chow groups
$$c\ell_{p, q}^H: CH^{d-q}(X, p-2q)_\Q\to \cA (\un,H_{p, q}(X))$$  
induced by $r_{p, q}^H$  and for $X$ smooth there is a cycle class map
$$c\ell^{p, q}_H : CH^{q}(X, 2q-p)_\Q\to\cA (\un,H^{p, q}(X))$$
induced by $r^{p, q}_H$ such that $c\ell^{2d-p, d- q}_H\cong c\ell_{p, q}^H$ under duality. 
\end{lemma}
\begin{proof} Write $p=2q+r$. The map $r^{p, q}_H$ is the mapping on Hom sets 
$$ H^{p,q}_M(X)=  \DM_\gm (M(X)[-r], \Z (q)[2q])\to \cA^{(\Z)} (H^{*}(\Z(q)[2q]), H^{*-r}(X))$$
given by applying $H^*$ since $H^*(\Z(q)[2q])$ is $\un (-q)$ concentrated in degree $2q$ as a graded object, the target is actually equal to
$ \cA (\un (-q) , H^{p}(X))$ since $*-r=2q$ for $*=p$. Similarly, for $r_{p, q}^H$ we have 
$$H_{p,q}^{\rm BM} (X)=\DM_\gm (M^c(X)^\star [r], \Z(-q)[-2q])\to \cA^{(\Z)} (H^*(\Z(-q)[-2q]), H^{*+r}(M^c(X)^\star))$$
where now $H^*(\Z(-q)[-2q])$ is $\un (q)$ in degree $-2q$ whence $*+r=-2q$ for $*=-p$.

Under the hypotheses of equidensionality, we also have $H_{p,q}^{\rm BM} (X)\cong CH^{d-q} (X,p-2q)$ so that we also get the claimed cycle map $c\ell_{p, q}^H$. 
For $X$ smooth, $c\ell^{p, q}_H$ is induced by $r^{p, q}_H$ under the identification $H^{p, q}_{\rm M}(X)\cong CH^{q}(X, 2q-p)_\Q$ and then apply Lemma \ref{dual} for the claimed compatibility. 
\end{proof}
For a mixed Weil cohomology $(\cA, H)$ the restriction $H\iota$ along the functor $\iota$ in \S \ref{1.1} of the cohomological functor $H$ yields a covariant functor 
  $$(H\iota)^*: \cM_\rat \to \cA^{(\Z)}$$ 
which is also a strong tensor functor: in Definition \ref{mixweil} the point axiom (Definition \ref{kunpdf}) corresponds to the unitality and the K\"unneth axiom (Definition \ref{kunpdf}) is strong monoidality as explained in Lemma \ref{kunp}. Moreover, we have that $(H\iota)^*$ sends $\cM_\rat^\eff$ to $\cA^{(\N)}$  via $\iota^\eff$ by effectivity (Definition \ref{bound}); finally, we also have $\Tr: H^2(\iota\L) \cong H^2(\P^1_k) \iso L$, $H^0(\iota h(X))\iso H^0(\iota h(\Spec (k)))$ if $X\in\SmProj_k$ is geometrically connected and $(H\iota)^*(\L)$ is concentrated in degree $2$ by the trace axiom (Definition \ref{trace}).

Refering to \cite{BVKW} for the notion of (generalised) Weil cohomology we let $(\cW_\ab, W_\ab)$ be the (abelian valued) universal Weil cohomology constructed in \cite[Thm.\,5.2.1 \& Cor.\,5.2.2]{BVKW}.
\begin{lemma} \label{weil}
 Let $(\cA, H)$ be a mixed Weil cohomology. The restriction of $H$ via $\iota$ yields a Weil cohomology $(\cA, H\iota)$ with cycle map 
      $$c\ell^{q}_H := c\ell^{2q, q}_H : CH^{q}(X)_\Q\to\cA (\un,H^{2q}(X)(q))$$
    given by Lemma \ref{reg}, for $q\in \N$ and $X\in\SmProj_k$.  Therefore, there is an exact strong $\otimes$-functor 
    $$\Phi_H: \cW_\ab\to \cA$$ such that the named restriction $(\cA, H\iota)$ is the push-forward of $(\cW_\ab, W_\ab)$ along $\Phi_H$.
\end{lemma}
\begin{proof}
    Just follows from \cite{BVKW}: apply Proposition 4.4.1 in  \loccit to $(\cA, H\iota)$ to get a Weil cohomology (in the sense of Definition 4.2.1) and then Theorem 5.2.1 to get the classifying functor $\Phi_H$ as claimed.
\end{proof}
\begin{example} \label{ex}
Let $(\cT\otimes, \un)$ be rigid $\otimes$-triangulated together with $(\cT^{\leq 0},\cT^{\geq 0})$ a $t$-structure. Let $\cA:=\cT^\heartsuit$ denotes the heart of $\cT$ and the $p$-th homology functor $H^p_t(C):= \tau_{\leq 0}\tau_{\geq 0}(C[p])$ of the heart. Assume that the $t$-structure is bounded and conservative, in the sense that $H^*_t: \cT\to \cA^{(\Z)}$ and that the zero object is the only object $T\in\cT$  such that 
$H^*_t(T)=0$. Assume that $\cA \otimes \cA\subseteq \cA$, or rather the tensor structure is compatible with the  $t$-structure, in such a way that the heart $(\cA, \otimes, \un)$ is $\otimes$-abelian and rigid. For example, $\cT:= D^b(\cA)$ for $\cA\in\Ex^\rig$. Let $R$ be a triangulated functor $$R:\DM_\gm\to \cT$$ that is also symmetric (strong) monoidal functor. Set $$R\Gamma(M):= R(M)^\vee= R(M^\star)$$ and denote the $i$-homologies $H^i_R := H^i_tR\Gamma$. Then $H^*_R$ provide a bounded cohomological functor endowed with an external product
$$H^*_R: \DM_\gm^\op \to \cA^{(\Z)}$$ 
 for which we also have that the K\"unneth formula holds (Definition \ref{kunpdf})
 $$\sum_{i+j=k} H^i_R(M)\otimes  H^j_R(N) \iso H^k_R(M\otimes N)$$
 since 
 $$\kappa^{k}_{M,N}: \sum_{i+j=k} H^i_t(R(M^\star))\otimes  H^j_t(R(N^\star))\iso H^k_t(R(M^\star)\otimes  R(N^\star)) \iso H^k_t(R((M\otimes N)^\star))$$
 where the first iso is given by the K\"unneth formula in \cite[Def.\,3.2, Thm.\,4.1 \& Cor.\,4.4]{BT} and the second is given by the strong monoidality of $R$. Now we also have that $\un\iso H^0_R(\Z)$ by the unitality of $R$ and $H^i_R(\Z) =0$ for $i\neq 0$ so that $(\cA, H_R)$ satisfies the K\"unneth and point axioms (Definition \ref{bound}).  Now pick $L\in\cA$ a Lefschetz object. Here $\un (q) = T^{\otimes q}$ where $T:= L^{-1}=L^\vee$ and $C\leadsto C(1):= C\otimes T$ shall be an auto-equivalence of $\cT$ which is compatible with the $t$-structure. Finally, for the trace axiom (Definition \ref{trace}), we should have $H^0_R(\Z)\iso H^0_R(X)$ for $X$ geometrically connected, $\Tr : H^2_R(\Z (1))\iso L=\un (-1)$ induced by $\Tr_R : R (\Z (1))\to L[-2]$ such that $R (\Z (q)[2q])\cong \un (q)$ for all $q\in \Z$ and $H^*_R (\Z (q)[2q])= \un (-q)$. If we also have that $H^i_R(X)= 0$ for $i\notin [0,2d]$, $X\in \Sm_k$ $d= \dim (X)$ we thus obtain that $(\cA, H_R)$ is effectively bounded (Definition \ref{bound}). 
 
 This is the case of Cisinsky-D\'eglise realisation functors for mixed Weil cohomologies in \cite[Thm.\,3 \& Thm.\,2.7.14]{CD} where $\cT= D^b(\cA)$ and $\cA$ is the category of finitely generated $K$-modules for $K$ a field of zero characteristic or, more generally, for $K$ an absolutely flat $\Q$-algebra as for Ayoub's new Weil cohomology, see \cite[Ex.\,3.2]{BV} and \cite{AW}. In particular, for classical Weil cohomologies such as $\ell$-adic, Betti and de Rham cohomology we have such a realisation functors, $R_\ell$, $R_{\rm Betti}$ and $R_{\rm dR}$. In characteristic zero, we also have Huber's enrichment of these classical realisations given by the mixed realisation $R_{\rm MR}$ \cite[Thm.\,6.3.15]{HMN} as well as $R_{\rm dR-Betti}$ the de Rham-Betti realisation, see \cite[Rmk.\,6.3.4]{HMN}. The corresponding regulators are those considered in \cite{ABVB}.
\end{example}

\subsection{Representability} For $(H, \kappa, \upsilon, \Tr)$ and $(H', \kappa', \upsilon', \Tr')$ mixed Weil cohomologies taking values in $(\cA, L)$, a morphism $\nu : (H, \kappa, \upsilon, \Tr) \to (H', \kappa', \upsilon', \Tr')$ is a monoidal natural transformation $\nu:H\to H'$ which is also compatible with $\Tr$ and $\Tr'$. We denote by $\MW (\cA, L)$ the category of mixed Weil cohomologies with values in $(\cA, L)$.

Let $\Ex_*^\rig$ be the 2-category of $\otimes$-abelian rigid $\Q$-linear pointed categories $(\cA, L)$, by a choice of a Lefschetz object, and exact strong tensor functors preserving the Lefschetz objects (as in \cite[Def.\,5.1.1]{BVKW}). Pushing forward along $F: (\cA, L)\to (\cA', L')$ in $\Ex_*^\rig$ we get
$$F_*:\MW (\cA, L)\to \MW (\cA', L')$$
where $F_*(H):=  FH$. We can make up a strong 2-functor
$$\MW(-) : \Ex_*^\rig \to \Cat$$
sending $(\cA, L)$ to  $\MW (\cA, L)$.
We have:
\begin{thm}[Universal mixed Weil cohomology] \label{mixed}
The 2-functor $\MW(-)$ is 2-representable, \ie  there is a cohomological functor
$$MW:\DM_\gm^\op\to \cMW$$ 
which is a mixed Weil cohomology $(\cMW, MW)$ such that for any mixed Weil cohomology $(\cA, H)$ there exists an exact strong $\otimes$-functor $$\Psi_H:\cMW\to \cA$$ such that $(\Psi_H)_*(MW)\cong H$. Moreover, the functor $\Psi_H$ is the unique extension of $\Phi_H$ (Lemma \ref{weil}) via $\Phi: \cW_\ab\to \cMW$ induced by the Weil cohomology $(\cMW, MW\iota)$ given by the restriction of $MW$ via $\iota$ as depicted in the following diagram
$$\xymatrix{\cW_\ab\ar[r]^{\Phi_H} \ar[d]_{\Phi} & \cA\\
\cMW\ar[ur]_{\Psi_H} & }$$
which is commutative up to (unique) natural isomorphism. Finally, the regulator $r^{p,q}_H$ (Lemma \ref{reg}) factors through the universal regulator $r^{p,q}_{MW}$ via ${\Psi_H}$.
\end{thm}
\begin{proof} Let  $\cT:=\DM_\gm^{\eff,\op}$ for short. The mixed Weil cohomology $MW$ is constructed in several steps.

{\it Step 1.} We construct $U^*:\cT\to \cU^{(\Z)}$ and $(U^*,\delta^*, \upsilon^*)$ an effectively bounded cohomological functor endowed with an external product on $\cT$ (Definition \ref{bound}) as follows. Applying Levine's universal external product (= universal unital symmetric lax tensor functor, see \cite[Thm.\,3.1.1]{BVKW}) to $\cT$ we get 
$$ (K,\delta_\kappa, \upsilon_\kappa) : \cT\to \cT^\kappa$$ providing a symmetric monoidal category $(\cT^\kappa, \otimes_\kappa, \omega)$ together with a universal external product 
\[\delta_\kappa :K(M)\otimes_\kappa K(N)\to K(M\otimes N)\]
and $\upsilon_\kappa:\omega\to K(\Z)$ for $M,N\in \cT$. Let $\cT^{ \kappa,\add}$ be the relative $\Q$-linear additive completion (see \cite[Prop.\,3.4.3]{BVKW}), in such a way that the functor $\cT^{ \kappa}\to \cT^{\kappa,\add}$ is a strong tensor functor and its composition $K^\add$ with $K$ is $\Q$-linear.  Then consider $\lambda: \cT^{\kappa,\add}\to T (\cT^{\kappa,\add})$ where $T$ is the 2-functor in \cite[Prop.\,5.4]{BVK} providing the universal $\otimes$-abelian category and the functor $S$ further composing with $K^\add$. We thus have  $(T (\cT^{\kappa,\add}), \otimes, \un)$ where $T (\cT^{\kappa,\add})$ is $\Q$-linear abelian, $\otimes$ is exact, $\un =T(\omega)$, and $S: \cT \to T (\cT^{\kappa,\add})$ is just a $\Q$-linear functor endowed with an external product in the sense of Definition \ref{cohext}. To make $S$ an effectively bounded cohomological functor, let $\cU$ be the quotient of $T (\cT^{\kappa,\add})$ by the minimal Serre tensor ideal such that for any distinguished triangle $M \to N \to F \longby{+1}$ and the corresponding complex $S(M) \to S(N) \to S(F)$ in $T (\cT^{\kappa,\add})$ its homology $S^\circ=0$ in $\cU$, also $S(M(X)[p])=0$ in $\cU$ for $p> 2d$ or $p<0$ where $X$ is smooth and equidimensional of $\dim (X)=d$. 
Since $\cT$ is generated by such $M(X)$ we thus get an effectively bounded cohomological functor $U$ given by composition 
$$\xymatrix{\cT\ar[r]_{K^\add} \ar@/^1.7pc/[rrr]^{U}\ar@/_1.9pc/[rr]_{S} & \cT^{\kappa,\add}\ar[r]_{\lambda} &T (\cT^{\kappa,\add}) \ar[r]_{q} & \cU}$$
which is also endowed with an external product $(U,\delta, \upsilon)$ induced by $(K,\delta_\kappa, \upsilon_\kappa)$ via the strong $\otimes$-functors $\lambda$ and $q$ (see \cite[Prop.\,4.5]{BVK} for properties of this latter quotient functor).

{\it Step 2.} We now obtain $MW^\eff:\cT\to \cMW^\eff$ satisfying K\"unneth, point axioms (Definition \ref{kunpdf}) and, partly, the trace axiom (Definition \ref{trace}) by taking a further quotient $\cU\to \cMW^\eff$ and $MW^\eff$ is its composition with $U$. In fact, we have $\upsilon: \un \to U (\Z)$, $\pi: U (\Z)\to U (X)$ for X geometrically connected and
\[\delta^k: \bigoplus_{i+j=k} U^i (M)\otimes U^j(N)\to U^k(M\otimes N)\]
 in $\cU$, for $i,j,k\in \Z$, $M, N\in\cT$. Let $\cMW^\eff$ be the ($\Q$-linear tensor) quotient of $\cU$ making the following list of morphisms invertible: $\upsilon$, $\pi$, $\delta$ and  $U^p(\Z(1))\to 0$ for $p\neq 2$ (note that $U^p(\Z)=0$ for $p\neq 0$ by the previous step).
 
{\it Step 3.} Finally, we make $L:= MW^\eff(\Z(1)[2])$ invertible obtaining a $\otimes$-abelian $\cMW:= \cMW^\eff[L^{-1}]$ (endowed with the tautological $\Tr$)  together with an exact strong $\otimes$-functor $\cMW^\eff\to \cMW$.  Voevodsky's condition that the permutation involution for the Tate object is the identity is verified in $\cT$ by \cite[Cor.\,2.1.5]{V} whence in $\cMW^\eff$ via K\"unneth formula. The claimed mixed Weil cohomology $(MW,\kappa,\upsilon)$ with values in $(\cMW, L)$ is given by further composition observing that since now $L$ is invertible in $\cMW$ the functor $MW^\eff$ lifts to $$MW^*: \DM_\gm^\op \to \cMW^{(\Z)}$$
 which is a strong tensor cohomological functor by construction. Its image is finitely supported by \cite[Lemma 3.3.6 a)]{BVKW} as $\DM_\gm$ is rigid.  Finally, it is clearly universal for mixed Weil theories taking values in $\Q$-linear $\otimes$-abelian categories by construction. Since $\DM_\gm$ is rigid therefore $MW$ takes values in $(\cMW_\rig, L)$ by Lemma \ref{kunp} but $\cMW_\rig$ is $\otimes$-abelian: this implies that $\cMW_\rig=\cMW$ by universality.
\end{proof} 

Recall from \cite[Def.\,8.3.4]{BVKW} that we also have a notion of tight Weil cohomology: there is a corresponding version in the mixed case.
\begin{defn}[Mixed tight Weil cohomology]  \label{mixtightweil}
Let  $(\cA, L)$ be a $\otimes$-abelian category together with a Lefschetz object.
A mixed Weil cohomology $(H, \kappa, \upsilon, \Tr)$ with values in $(\cA, L)$ is \emph{tight} if the following additional conditions are satisfied:
\begin{enumerate}
    \item \protect(Weak Lefschetz): we have $H^p(U)=0$ for affine $U= \Spec (R)\in \Sm_k$ and $p>\dim U$
    \item \protect(Hard Lefschetz): for $p\le d =\dim (X)$ the ismorphism
$$L^{p}: H^{d-p}(X)\iso H^{d+p}(X)(p)$$
induced by the Lefschetz operator $L:= L_Y$ for $X\in \SmProj_k$ and $Y$ a smooth hyperplane section of $X\in\SmProj_k$
    \item \protect(Albanese invariance): $H^1(\Alb_X)\iso H^1(X)$ induced by the canonical map $M(X)\to \Alb_X$ given by Serre Albanese variety for $X\in\Sm_k$
    \item \protect(Normalisation): $H^0(\pi_0(X))\iso H^0(X)$ induced by the canonical map $M(X)\to \pi_0(X)$ for the scheme of constants $\pi_0(X)$ and $X\in\Sm_k$.
\end{enumerate}
 \end{defn}
 \begin{remark}\label{van}
     In Definition \ref{mixtightweil}, for a smooth hyperplane section $Y$ of $X\in\SmProj_k$ such that $\dim(X)=d$,  from (1) we get that $H^p(X,Y)=0$ if $p< d$ and $H^p(X)\iso H^p(Y)$ if $p< d-1$ as usual by Lemma \ref{rel}. The injectivity of $H^{d-1}(X)\into H^{d-1}(Y)$ follows from (2). For $H^p_R (X)$ given by an object $R\Gamma (M(X))$ of $\cT$ as in Example \ref{ex} and provided with a Lefschetz morphism $R\Gamma (M(X))[-1]\to R\Gamma (M(X))[1](1)$ inducing the isomorphisms in (2) we have a decomposition $$R\Gamma (M(X))\cong \oplus H^p_R (X)[-p]$$ in $\cT$ by Deligne's decomposition theorem, see \cite{B} for a simple proof of the latter. 
 \end{remark}
 
Let $\MW^+ (-)$ be the 2-functor of mixed tight Weil cohomologies. The analogue of Lemma \ref{weil} holds for tight. Let $(\cW_\ab^+, W_\ab^+)$ be the (abelian valued) universal tight Weil cohomology constructed in \cite[Thm.\,8.4.1]{BVKW}.
\begin{lemma}\label{pure}
   If $H\in \MW^+ (\cA, L)$ then the restriction of $H$ via $\iota$ yields a tight Weil cohomology $(\cA, H\iota)$.  Therefore, there is an exact strong $\otimes$-functor 
    $$\Phi_H^+: \cW_\ab^+\to \cA$$ such that the restriction $(\cA, H\iota)$ is the push-forward of $(\cW_\ab^+, W_\ab^+)$ along $\Phi_H^+$.
\end{lemma}
\begin{proof}
    By Definition 8.3.4 in \cite{BVKW} and the previous Remark \ref{van} we have that $(\cA, H\iota)$ is tight and  Theorem 8.3.4 in \loccit yields the claimed classifying fucntor $\Phi_H^+$ as claimed.
\end{proof}

\begin{thm}[Universal mixed tight Weil cohomology] \label{tight}
The 2-functor $\MW^+ (-)$ is 2-representable by $(\cMW^+, MW^+)$. For any tight $H\in \MW^+ (\cA, L)$ the classifying functor $\Psi_H$ factors through $\cMW^+$ which is a quotient of $\cMW$ (Theorem \ref{mixed}) fitting in the following commutative diagram 
$$\xymatrix{\cW_\ab\ar[d]_{}\ar[r]_{\Phi} \ar@/^1.8pc/[rr]^{\Phi_H} &\cMW\ar[r]^{\Psi_H} \ar[d]_{} & \cA\\
\cW_\ab^+ \ar[r]_{\Phi^+} \ar@/_3.5pc/[rru]_{\Phi_H^+}  &\cMW^+\ar[ur]_{\Psi_H^+} & }$$
where $\Psi_H^+$ is the classifying strong tensor functor such that $(\Psi_H^+)_*(MW^+)\cong H$.
\end{thm}
\begin{proof}
As for \cite[Thm.\,8.3.4]{BVKW} describing $\cW_\ab^+ $ as a quotient of $\cW_\ab$, we let $\cMW^+$ be the quotient of $\cMW$ making invertible $H^p(U)\to0$ for affine $U\in \Sm_k$ and $p>\dim U$, $L^{p}: H^{d-p}(X)\to H^{d+p}(X)(p)$ for $p\le \dim X$ and $X\in \SmProj_k$,  $H^1(\Alb_X)\to H^1(X)$ and $H^0(\pi_0(X)) \to H^0(X)$ and $X\in \Sm_k$. The claimed commutative diagram follows from Lemmas \ref{pure} and \ref{weil}.
\end{proof}
\section{Comparison with Nori and Andr\'e motives}
\subsection{Universal enrichment} To compare Nori's construction with ours we first introduce the corresponding analogue in this general context of universal cohomology.
For any mixed Weil cohomology $H\in \MW (\cA, L)$, from the Theorem \ref{mixweil}, we obtain $\Psi_H:\cMW \to \cA$ and we then get the $\Q$-linear $\otimes$-abelian category
$$\cMW_H := \cMW/\ker \Psi_H$$ 
together with a Lefschetz object $L_{H}$ induced by the projection. There is a faithful exact strong tensor functor 
$$\bar\Psi_H: \cMW_H \into \cA$$ 
such that $MW_H\in \MW (\cMW_H, L_H)$ is induced by $MW$ via the pushforward along the projection: this  yields $\bar\Psi_H MW_H\cong H$ by construction.
If $H\in\MW^+ (\cA, L)$ is tight then $\Psi_H$ factors through $\cMW^+$ by Theorem \ref{tight}, we have
$$\cMW_H \cong \cMW^+/\ker \Psi_H^+$$
as well and $\bar\Psi_H^+ MW_H^+\cong H$. Note that for the universal theory $H=MW$ we have that $\cMW_H = \cMW$ is a tautology.
\begin{defn}\label{equi} 
Call $(\cA', H')$ an \emph{enrichment} of $(\cA, H)$ if there is a \underline{faithful} exact strong tensor functor $F: \cA'\to\cA$ such that  $(\cA, H)$ is the pushforward of $(\cA', H')$ along $F$ \ie $ H\cong F H'$ compatibly with the Lefschetz objects, and we also say that $(\cA, H)$ is a \emph{realisation} of $(\cA', H')$ in this case. Say that two mixed Weil cohomologies $(\cA, H)$ and $(\cA', H')$ are \emph{equivalent} if $\ker \Psi_H=\ker \Psi_{H'}$.
\end{defn}
The following is the mixed analogue of \cite[Thm.\,6.1.7]{BVKW}. 
\begin{thm}\label{initial}
    The Weil cohomology $(\cMW_H, MW_H)$ is the universal or initial enrichment of the mixed Weil cohomology $(\cA, H)$. If $(\cA, H)$ is tight then $(\cMW_H, MW_H)$ is tight and universal among mixed tights.  Moreover, the tight Weil cohomology $(\cA, H\iota)$ and its universal enrichment   
     $\bar \Phi_H:\cW_H^\ab\into \cA$ are fitting in the following commutative diagram 
    $$\xymatrix{\cW_\ab^+\ar[d]_{\Phi^+}\ar[r]_{} &\cW_H^\ab\ar[d]_{\Xi_H} \ar[dr]^{\bar \Phi_H} &\\
\cMW^+\ar[r]_{}\ar@/_1.7pc/[rr]_{\Psi_H^+}  &\cMW_H \ar[r]_{\bar \Psi_H} & \cA}$$
where $\Xi_H : \cW_H^\ab\into \cMW_H $ is an induced faithful exact $\otimes$-functor. If $H$ is classical then $\cMW_H$ is Tannakian. 
\end{thm}
\begin{proof} The universality of $(\cMW_H, MW_H)$ comes from the universal property of the quotient and its tightness is given by the faithfulness of $\bar\Psi_H$ whenever $(\cA, H)$ is tight. As from the proof of \cite[Thm.\,6.1.7]{BVKW} the category $\cW_H^\ab:= \cW_\ab/\ker \Phi_H$ for $\Phi_H$ as in Lemma \ref{pure}, $\Phi^+$ is induced by a localisation of $\Phi$ as in Theorem \ref{tight}, and $\Phi$ is sending $\ker \Phi_H$ to $\ker\Psi_H$ therefore we get $\Xi_H$ and the commutative diagram as claimed. Faithfulness of $\Xi_H$ follows from faithfulness of $\bar \Phi_H$.
\end{proof}

\subsection{Nori and Andr\'e motives} Let's apply Theorem \ref{initial} to $H= H_{\rm Betti}$ the mixed tight Weil cohomology given by Betti cohomology for $k$ a subfield of the complex numbers. 

Recall the construction of the $\Q$-linear $\otimes$-abelian rigid category of Nori motives $\cNM$ (see \cite[\S 4]{BVHP}, \cite{AN} and \cite[Thm.\,9.3.10]{HMN}). The category of effective cohomological Nori motives $\cNM^\eff$ is the universal abelian category associated with the Nori quiver $D^{\rm Nori}$ with vertices $(X, Y, i)$ with $i\in\N$ given by (good) pairs $(X,Y)\in\Sch_k^\square$ and its Betti cohomology representation in finite dimensional $\Q$-vector spaces.  Let $H_{\rm Nori}^i(X, Y)\in \cNM^\eff$ be Nori's representation. Because of Nori's basic Lemma the category $\cNM^\eff$ inherits a tensor structure. The rigid category $\cNM$ is then obtained from $\cNM^\eff$ by tensor inverting the Lefschetz object $L_{\rm Nori}:= H_{\rm Nori}^2(\P^1)\cong H^1_{\rm Nori}(\G_m)$ (see \cite[Def.\,9.3.7]{HMN}). Finally, Nori's basic Lemma also inhanced Betti cohomology with cellularity and a contravariant triangulated strong $\otimes$-functor
$$R_{\rm Nori}: \DM_\gm^\op\to D^b(\cNM)$$
lifting the contravariant $R_{\rm Betti}$ along the canonical forgetful functor from $D^b(\cNM)$ to $D^b(\Q)$,  the bounded derived category of finite dimensional $\Q$-vector spaces (see \cite[Thms.\,9.1.5, 10.1.1 \& 10.1.4]{HMN}). As in the Example \ref{ex} the functor $R_{\rm Nori}$ yields
$$H_{\rm Nori}:= H^0R_{\rm Nori}: \DM_\gm^\op\to \cNM$$
which is a mixed (tight) Weil cohomology (it is tight as an enrichment of Betti which is tight). There is a covariant homological version of $R_{\rm Nori}$ taking values in the category of Nori homological motives (see \cite[Prop.\,7.12]{CG}): the two are interchanged via duality. 

Recall that the restriction of Betti cohomology to smooth projective varieties yields a tensor equivalence
$$\theta_H: \cM^A_H\iso \cW_{\rm Betti}:= \cW_H^\ab$$
with the semi-simple abelian category $\cM^A_H$ of Andr\'e motives associated with $H= H_{\rm Betti}$, as proven in \cite[Thm.\,9.3.3]{BVKW}, and
$$\Xi_H: \cW_{\rm Betti}\into \cMW_{\rm Betti}:= \cMW_H$$
by Theorem \ref{initial}. On the other hand, the description of André motives as pure Nori motives is provided by \cite[Thm.\,5.5]{AN} or \cite[Thm.\,10.2.7]{HMN}.
Summarizing, we obtain the following:
\begin{thm}\label{Nori}
 We have a comparison (exact, strong) tensor equivalence 
$$\bar \Psi_{\rm Nori}: \cMW_{\rm Betti}=\cMW_{\rm Nori}\iso \cNM$$
where $\cMW_{\rm Nori}:= \cMW_H$ for $H= H_{\rm Nori}$. The precomposition of $\bar\Psi_{\rm Nori}$ with $\Xi_H\theta_H$ identifies Andr\'e motives $\cM^A_H$ with the full abelian tensor subcategory generated by $H^i_{\rm Nori}(X)$ with $X\in \SmProj_k$ and $i\in\N$. 
\end{thm}
\begin{proof} 
The forgetful tensor functor from $\cNM$ to $\Q$-vector spaces is lifting $H_{\rm Betti}$ to  $H_{\rm Nori}$ (by Nori's construction) making $H_{\rm Nori}$ an enrichment of $H_{\rm Betti}$ so that $H_{\rm Betti}$ is equivalent to $H_{\rm Nori}$ (in the sense of Definition \ref{equi}) and $\cMW_{\rm Betti}=\cMW_{\rm Nori}$.

The classifying functor $\Psi_{\rm Nori}: \cMW\to \cNM$ corresponding to $H_{\rm Nori}$ by Theorem \ref{mixed} factors through $\bar \Psi_{\rm Nori}: \cMW_{\rm Nori}\into \cNM$ by construction. Moreover, there is a representation (in fact, a graded tensor representation on good pairs, see \cite{BVHP}) of $D^{\rm Nori}$ in $\cMW_{\rm Nori}$ providing a faithful exact tensor functor $\cNM\into \cMW_{\rm Nori}$ by the universal property of $\cNM$. Thus the latter functor is a quasi-inverse of $\bar \Psi_{\rm Nori}$ by Theorem \ref{initial}. 
\end{proof}

\begin{remark}\label{basic}
Let's point out that for any mixed Weil cohomology $(\cA, H)$ the $\otimes$-abelian rigid category $\cMW_H$ exists (without appealing to the basic lemma). Moreover,  $\cMW_H$  is equipped with a universal representation of Nori's quiver $D^{\rm Nori}$ which is an enrichment of that of $H$ in $\cA$ by setting $$H: D^{\rm Nori}\to \cA \hspace{1cm} (X,Y,i) \mapsto H^i(X, Y)$$ where $H^i(X, Y)$ is relative cohomology in Lemma \ref{rel}. 
In characteristic zero, for $H$ tight one can see that this relative cohomology is cellular (see \cite[Def.\,1.3.1]{BVPT}) with respect to $\Sch_k^\square$ the category of pairs. In fact, Nori's direct proof of basic Lemma reported in \cite[Thm.\,2.5.1, \S 2.5.2 \& Lemma 2.5.8]{HMN} applies to any such $H$ by excision (Lemma \ref{rel}), relative duality (Remark \ref{reldual}) and weak Lefschetz (Definition \ref{mixtightweil} and Remark \ref{van}). Thus relative cohomology objects of smooth affines schemes can be tautologically computed by the cellular complex. Moreover, this applies to the universal cohomology and we may expect cellularity 
in positive characteristics as well. For example, that of $\ell$-adic cohomology shall also be granted by Beilinson's basic Lemma, \cf \cite[Thm.\,2.5.7]{HMN}.
\end{remark}

\section{Mixed versus pure motives}
\subsection{Further properties} We can get a universal mixed Weil cohomology relatively to any class of mixed Weil cohomologies. Let $\cS$ be a class of mixed Weil cohomologies. For $(\cA,H)\in \cS$ we get $\Psi_H:\cMW\to \cA $ given by Theorem \ref{mixed}. Let
$$\cI_\cS:= \bigcap_{H\in \cS}\ker \Psi_H  \ \
\text{and}\ \ \cMW_\cS:= \cMW/\cI_\cS$$
so that we obtain an induced $MW_{\cS}$, pushforward of $MW$ (= the universal theory in Theorem \ref{mixed}) along the projection from $\cMW$ to $\cMW_\cS$. Moreover, for $(\cA,H)\in \cS$, each $MW_H$ of Theorem \ref{initial} is the push-forward of $MW_{\cS}$ along the further quotient $\cMW_\cS\to \cMW_H$. Therefore, we get a functor $$\rho_H: \cMW_\cS\to \cA$$ refining $\Psi_H$, so that $H$ is the push-forward of $MW_\cS$ along $\rho_H$.
This mixed Weil cohomology $(\cMW_{\cS},MW_{\cS})$ is universal relatively to the class $\cS$.
Actually, the pattern indicated in \cite[Thm.\,3.4]{BV} is fulfilled and we easily obtain the following key fact (analogue of \cite[Thm.\,6.6.3]{BVKW}).
\begin{thm}\label{mixmot}  
For $\cS$ containing classical mixed Weil cohomologies, the following conditions are equivalent:
\begin{enumerate}
\item $\cMW_\cS$ is connected, \ie $Z_\cS:= \End (\un)$ is a domain
\item  all $H\in \cS$ are realisations of $MW_\cS$ along $\rho_{H}$  (see Definition \ref{equi})
\item  all $H\in\cS$ are equivalent (see Definition \ref{equi})
\item if $H\in\cS$ then $\cMW_\cS\iso \cMW_H$ is a tensor equivalence
\item $\cMW_\cS$ is Tannakian.
\end{enumerate}
\end{thm}
Since $\cMW_\cS$ is rigid if $Z_\cS$ is a domain then is a field. Actually, $Z_\cS$ is a domain if and only if $\un$ is simple. In fact, subojects $U$ of $\un$ correspond to splittings $\un\cong U\oplus U^\perp$ whence to idempotents $e\in Z_\cS$ (see \cite[Prop.\,3.2]{BVK}). Parallel to a theory of pure motives as in \cite[\S 3]{BV} we set:
\begin{defn}
    Say that a \emph{theory of mixed motives} exists for $\cS$ if  $\cMW_\cS$ is connected.
\end{defn}
This can be translated by saying that for such a class the universal cohomology of the point is simple. In characteristic zero, a theory of mixed motives exists for classical Weil cohomologies: it coincides with Nori motives by Theorem \ref{Nori} since all classical Weil cohomologies are comparable and therefore equivalent. For $\cS$ identified with the class of mixed tight Weil cohomologies we have that $\cMW_\cS=\cMW^+$. Akin to the standard hypothesis that $\cW_\ab^+$ is connected \cite[Hyp.\,3.10]{BV} or Ayoub’s conjecture \cite[Conj.\,3.20]{AW} we have:
\begin{hyp}\label{standard}
A theory of mixed motives exists for the class of mixed tight Weil cohomologies: equivalently, $\cMW^+$ is connected.
\end{hyp}
\begin{remarks}\label{rmkstd}
a) The Hypothesis \ref{standard} is also equivalent to say that $\cMW^+\cong \cMW_H$ for any tight $H$ and/or to say that $\cMW^+$ is Tannakian (by Theorem \ref{mixmot} applied to the class of mixed tight Weil cohomologies). In particular, idependently of $\ell\neq \car(k)$, this implies that all $\ell$-adic cohomologies are equivalent and the Tannakian category $\cMW^+\cong \cMW_H$ for $H=H_\ell$ an $\ell$-adic cohomology can just be considered the $\ell$-adic analogue of Nori motives, see also Remark \ref{basic}.\\

b) Conversely, if $\cMW^+\cong \cMW_H$ for one $\ell$-adic cohomology $H=H_\ell$ then  $\cMW^+$ is Tannakian, the Hypothesis \ref{standard} holds true and the same is true for all primes $\ell\neq \car(k)$. In characteristic zero, for $H=H_{\rm Betti}$, $\cMW^+\cong \cMW_H$ is then also equivalent to the actual Nori motives $\cNM$ by Theorem \ref{Nori} and Nori motives are then universal for all mixed tight Weil 
cohomologies.\\

c) Without the Hypothesis \ref{standard}, the picture for any classical $H$ over a subfield of the complex numbers is the following
$$\xymatrix{\cW_\ab^+\ar[d]_{\Phi^+}\ar[r]_{} &\cW_H^\ab\ar[d]_{\Xi_H} & \ar[l]_{\simeq}^{\theta_H}\cM_H^A\ar[d]  \\
\cMW^+\ar[r]_{} &\cMW_H\ar[r]^{\simeq}_{\bar\Psi_{\rm Nori}} & \cNM & }$$
where the functor $\Xi_H$ is fully faithful. However, in arbitrary characteristic, only the left square of the diagram above is canonically defined for any $H$ and the fully faithfullness of $\Xi_H$ as well as that of $\Phi^+$ should be appropriately investigated. 
\end{remarks}

\subsection{Conjectural picture}
Let ${\rm W^+_\ab}(\cA, L)$ be the category of tight Weil cohomologies with values in $(\cA, L)$ and recall that the induced 2-functor ${\rm W^+_\ab}: \Ex_*^\rig \to \Cat$ is 2-represented by the universal tight Weil cohomology $(\cW^+_\ab, W_\ab^+)$ (see Theorem \cite[Thm.\,8.4.1]{BVKW}). 
For $H\in\MW^+(\cA, L)$ denote $H_{\rm pure}:= H\iota\in {\rm W^+_\ab}(\cA, L)$ the restriction via $\iota$ of $H$ as in Lemma \ref{pure}. We have:
\begin{lemma} \label{purification}
The restriction or purification functor 
$$\Pi_{(\cA, L)}: \MW^+(\cA, L)\to {\rm W^+_\ab}(\cA, L)$$
sending $H$ to $H_{\rm pure}$ is natural in ${(\cA, L)}$. The cohomology $H_{\rm pure}$ takes  values in a subcategory $\cA_{\rm pure}\subseteq \cA$ with the same Lefschetz object $L\in\cA_{\rm pure}$, defined as the strictly full abelian tensor subcategory generated by the cohomology of smooth projective varieties, \ie generated by the objects $H^p(X)(q)$ for $X\in\SmProj_k$ and $p, q\in\Z$.
\end{lemma}
\begin{proof}
    Straightforward.
\end{proof}
For the universal mixed tight Weil cohomology of Theorem \ref{tight} we thus get the purification $MW^+_{\rm pure}$ with values in $\cMW^+_{\rm pure}$ as the restriction $MW^+_{\rm pure}:=MW^+\iota$ to $\cM_\rat$ and we also obtain that any purification $H_{\rm pure}\in{\rm W^+_\ab}(\cA, L)$ is associated with an exact tensor functor
$$\Psi_H^{\rm pure}: \cMW^+_{\rm pure}\to \cA_{\rm pure}$$
restriction of $\Psi_H^+$ providing $H_{\rm pure}$ as the push-forward of $MW^+_{\rm pure}$ along $\Psi_H^{\rm pure}$ but it is not clear that this functor is unique with this property. However, we have an exact tensor functor 
$$\Phi^+_{\rm pure}: \cW^+_\ab\to \cMW^+_{\rm pure}$$
which is a refinement of $\Phi^+$ in Theorem \ref{tight} and whose composition with $\Psi_H^{\rm pure}$ is the unique classifying functor given by the universal property of $(\cW^+_\ab, W_\ab^+)$. 

 Clearly, it will be agreable to show that $\Phi^+_{\rm pure}$ is an equivalence. Actually, for a mixed $H$ its purification $H_{\rm pure}$ should also be related with Bondarko's weight triangulated functor $t: \DM_\gm\to K^b(\cM_\rat)$ (see \cite[Prop.\,6.3.1]{BOW}) which is also symmetric monoidal (see \cite{KA}) in such a way that weight filtrations, as already noted in Remark \ref{weight}, should be playing a canonical rôle in purification. 
 
 Moreover, we then may expect that all extensions in $\cMW^+_{\rm pure}$ are splitting: this is to say that such a category is split in the sense of \cite[Def.\,5.2]{BVK} and it is equivalent to say that $\un$ (= the universal cohomology of the point) is projective (see \cite[Prop.\,5.5]{BVK}). This property is hinted by weight arguments and/or by the fact that $\DM_\gm^\eff$ is generated, as a triangulated category, by direct summands of $M(X)$ for $X\in\SmProj_k$ and any distinguished triangle with all three vertices being such $M (X)$ splits (see \cite[Cor.\,4.2.6]{V}). 
 
 Finally, in arbitrary charatceristic, it seems reasonable to expect that any mixed tight Weil cohomology is cellular as we have explained in Remark \ref{basic}. Following the pattern indicated by Nori we may then construct a symmetric monoidal triangulated functor
$$R_H^+: \DM_\gm^\op\to D^b(\cMW_H)$$
representing the universal enrichment of the (mixed) Weil cohomology via the cellular complex (dual construction of \cite[Prop.\,7.12]{CG}). Note that if this latter property holds true for the universal cohomology then it is also verified by any cohomology. Therefore, any mixed tight Weil cohomology shall be represented by such a triangulated symmetric monoidal functor $R_H$ as in the Example \ref{ex} and such that $R_H\Gamma (M(X)) \cong \oplus H^p(X)[-p]$ for $X\in\SmProj_k$ as in Remark \ref{van}. Summarizing up:
 \begin{conjecture} \label{conj}
 The following properties hold true:
 \begin{enumerate}
     \item \protect(Purity):  The functor $\Phi^+_{\rm pure}$ is a tensor equivalence. 
     \item \protect(Splitness):  The category $\cMW^+_{\rm pure}$ is split.
     \item \protect(Cellularity): $(\cMW^+,MW^+)$ is cellular.
 \end{enumerate}
 \end{conjecture}
These properties are apparently weaker than the existence of a theory of pure or mixed motives (for tight cohomologies).
\begin{propose} \label{prop}
If $\Phi^+$ is fully faithful, \eg under the purity Conjecture \ref{conj} (1), we have that the Hypothesis \ref{standard} is equivalent to the Standard Hypothesis \cite[Hyp.\,3.10]{BV}. Moreover, for $k$ a subfield of the complex numbers, these hypotheses imply the Conjecture \ref{conj} and identify the fully faithful functors $\Xi_H$ and $\Phi^+$ as follows
$$\xymatrix{\cW_\ab^+\ar[d]_{\Phi^+}\ar[r]^{\simeq} &\cW_H^\ab\ar[d]^{\Xi_H}  \\
\cMW^+\ar[r]_{\simeq} & \cMW^+_H}$$
for every $H$ tight. 
\end{propose}
\begin{proof} If $\Phi^+$ is fully faithful then the ring $Z_+:=\End (\un)$ is the same for both source $\cW^+_\ab$ and target $\cMW^+$ of $\Phi^+$.
Over a subfield of the complex numbers, if $Z_+$ is a domain then $Z_+=\Q$, $\cW^+_\ab\cong \cW^+_H$ and $\cMW^+\cong \cMW^+_H$ for any $H$ tight and $\Phi^+\cong \Xi_H$ too. Picking up $H=H_{\rm Betti}$ and refining the diagram in Remark \ref{rmkstd} c) we obtain that Conjecture \ref{conj} is verified by Theorem \ref{Nori} and the well known property that Andr\'e motives are pure Nori motives.
\end{proof} 

Grothendieck standard conjectures further imply that $\cW_\ab^+$ is given by Grothendieck motives as explained in \cite[Thm.\,3.8]{BV}. Over $k =\bar \Q$, by the way, the period or fullness conjecture in \cite[\S 1.3]{ABVB} or \cite[Conj.\,6.3]{HU} imply the Grothendieck standard conjectures.

\begin{remarks} The interested reader can certainly inspect more general constructions or suitable variants. For example:\\

a) for integral coefficients by considering the larger category of motivic \'etale sheaves $\DM_\et$ and cohomological additive functors with respect to arbitrary direct sums taking values in Grothendieck abelian right exact tensor categories;\\

b) for non homotopical invariant cohomologies or other variants of Voevodsky motives based on schemes with modulus or logarithmic motives \cite{LM} with respect to Fontaine-Illusie logarithmic geometry and log cohomologies such as log Betti, log \'etale and log de Rham \cite{LB};\\

c) for cohomology theories with coefficients in an integrable connection, as de Rham cohomology or rapid decay cohomology,  generalising exponential motives \cite{EM}.
\end{remarks}


\end{document}